\documentclass[12pt]{amsart}

\usepackage[pdftex,
colorlinks, 
bookmarksnumbered, 
bookmarksopen, 
linkcolor=blue, 
citecolor=blue, 
]{hyperref}
\usepackage{eucal}
\usepackage{mathptmx}
\usepackage{graphicx}
\usepackage{amssymb}
\usepackage{amsthm}
\providecommand {\norm}[1] {\lVert#1\rVert}
\providecommand {\bignorm}[1] {\Bigl\lVert#1\Bigr\rVert}

\providecommand {\abs}[1] {\lvert#1\rvert}

\providecommand {\set}[1]{\lbrace #1 \rbrace}

\providecommand {\app}[3]{\mathcal{E}^{#1}_{#2} ({#3})}
\providecommand {\inject}{\hookrightarrow}
\providecommand {\inv}[1]{{#1}^{-1}}

\newcommand {\bn} {\ensuremath{\mathbb{N}}}
\newcommand {\bno} {\ensuremath{\mathbb{N}_0}}

\newcommand {\bc} {\ensuremath{\mathbb{C}}}

\newcommand {\mA} {\ensuremath{\mathcal{A}}}

\newcommand {\mE} {\ensuremath{\mathcal{E}}}
\newcommand {\mF} {\ensuremath{\mathcal{F}}}

{

\newcommand {\mB} {\ensuremath{\mathcal{B}}}

\newcommand {\bop}  {{\ensuremath{\mathcal B ( \ell^2)}}}

\DeclareMathOperator{\cond}{cond}

\newcommand \smult {submultiplicative}

\newcommand \cont {continuous}

\newcommand \BA {Banach algebra}

\newcommand \as {approximation space}
\newcommand \IC {inverse-closed}

\newcommand \nci {norm-controlled inversion}

\newtheorem{prop}{Proposition} [section]
\newtheorem{cor}[prop]{Corollary}
\newtheorem{thm}[prop]{Theorem}

\newtheorem{lem}[prop]{Lemma}

\theoremstyle{definition}
\newtheorem{defn}[prop]{Definition} 

\theoremstyle{remark}
\newtheorem*{rem}{Remark}
\newtheorem*{rems}{Remarks}
\newtheorem{ex}[prop]{Example}

\usepackage{amsmath,amsfonts,amsthm,amsopn,cite,mathrsfs}
\usepackage{epsfig,verbatim}
\usepackage{subfigure}

\newcommand{\tfa}{time-frequency analysis}

\newcommand{\fif}{if and only if}

\newcommand{\psdo}{pseudodifferential operator}

 \theoremstyle{definition}

\newcommand{\beqa}{\begin{eqnarray*}}
\newcommand{\eeqa}{\end{eqnarray*}}

\newcommand{\field}[1]{\mathbb{#1}}
\newcommand{\bR}{\field{R}}        
\newcommand{\bN}{\field{N}}        
\newcommand{\bZ}{\field{Z}}        
\newcommand{\bT}{\field{T}}        %
 \def\cO{\mathcal{O}}

\def\rd{\bR^d}

\def\<{\left<}
\def\>{\right>}

\def\inv{^{-1}}

\def\mv1{M_v^1}

\newcommand{\vs}{\vspace{3 mm}}

\setlength{\topmargin}{-8mm}
\setlength{\headheight}{8pt}
\setlength{\textheight}{220mm}  

\setlength{\oddsidemargin}{0pt}
\setlength{\evensidemargin}{0pt}
\setlength{\textwidth}{148 mm}   

\newcommand{\dyad}{2^k}

\begin{document}
\begin{abstract}
Every    differential subalgebra of a unital $C^*$-algebra is spectrally
invariant.  We derive  a quantitative version of this well-known fact
and show that a minimal amount of smoothness, as given by a
differential norm, already implies norm control. We obtain an explicit
estimate for the differential norm of an 
invertible element $a$.  This estimate depends only on the condition
number of $a$ and the ratio of  two norms. 
\end{abstract}

\title{Norm-Controlled Inversion in Smooth Banach Algebras, I}
\author{Karlheinz Gr\"ochenig}
\address{Faculty of Mathematics \\
University of Vienna \\
Nordbergstrasse 15 \\
A-1090 Vienna, Austria}
\email{karlheinz.groechenig@univie.ac.at}
\author{Andreas Klotz}
\address{Faculty of Mathematics \\
University of Vienna \\
Nordbergstrasse 15 \\
A-1090 Vienna, Austria}
\email{andreas.klotz@univie.ac.at}

\date{}
\thanks{Both authors were   
 supported by the  project P22746-N13  of the
Austrian Science Foundation (FWF)}
\maketitle

\section{Introduction}

In many contexts one encounters the problem of inversion of smooth
elements in a Banach algebra. Often the smoothness is preserved under
inversion, and one is challenged to control the smoothness norm of the
inverse. This is the quantitative problem of \nci .

The standard example of norm-controlled inversion occurs already in calculus under
the name of the quotient rule. Let $ C^1(\bT )$ be the algebra of continuously differentiable functions
on the torus $\bT $ and $ C(\bT )$ be the algebra of all
continuous functions on the torus with the norms $\norm{f}_C =
\norm{f}_\infty = \max
_{t\in \bT } |f(t)|$ and $ \norm{f}_{C^1} = \norm{f}_\infty +
\norm{f'}_\infty $. The quotient rule $(1/f)' = -
f' / f^2$ leads to an obvious estimate for  the $C^1$-norm of  $1/f$,
namely,
\begin{equation}
  \label{eq:sont}
\norm{\tfrac{1}{f}}_{C^1} \leq \norm{\tfrac{1}{f}}_{\infty} + \norm{f'}_\infty
  \norm{\tfrac{1}{f}}_\infty ^2 = \Big(  \norm{f}_{C^1}
  \norm{\tfrac{1}{f}}_C + 1\Big) \norm{\tfrac{1}{f}}_C  \, .
\end{equation}
In other words, we can control the norm of the inverse of $f$ in the
subalgebra $C^1(\bT )$ by the norm of $f$ in $C^1(\bT ) $ and the norm of $1/f$ in
$C(\bT )$. This is the phenomenon of norm control in a nutshell. Can we
obtain quantiative estimates for 
the norm of the inverse in a subalgebra,  when we have access to
information about the inverse in an ambient algebra? Usually
invertibility and norm bounds are easier to obtain in the large
algebra.  

In applied mathematics this problem  is underlying the regularity of
solutions of linear systems in infinite-dimensional spaces. If a
matrix possesses off-diagonal decay and is invertible on $\ell ^2$,
then its inverse possesses the same quality of off-diagonal
decay~\cite{Bas90,Bas97,jaffard90,GL04a}. As a consequence the solution of the linear equation
$Ax=b$ with an input vector $b$ of a certain decay   possesses a
solution $x$ with the same decay.  This principle has
numerous and often fundamental  applications in numerical analysis~\cite{DFG10,GRS10,Lin06}, the theory of \psdo s~\cite{Sjo95,gro06}, in frame
theory~\cite{gro04,BCHL06}, in \tfa~\cite{GL04}, and in sampling theory~\cite{gro04,sun08a}. 

In non-commutative geometry   smooth subalgebras and their
spectral invariance  are an important technical tool for the computation of 
the $K$-group of an algebra. The density theorem asserts that the
$K$-groups of a smooth subalgebra coincide with the $K$-groups of the
original $C^*$-algebra~\cite{black06}. 

A general theory of abstract smoothness and an axiomatic construction
of inverse-closed subalgebras was developed in~\cite{GK10,Klo12}. With a minimum
amount of natural structures on a given Banach algebra, namely
unbounded derivations or commutative automorphism groups, one can
construct subalgebras of smooth elements where the smoothness mimicks
the Besov regularity or the H\"older-Lipschitz continuity of functions
on $\rd $. This theory establishes a new link between methods of
approximation theory and the theory of operator algebras. So far, this
construction of smooth subalgebras and their inverse-closed is a
purely \emph{qualitative} theory: if an element is smooth, then its
inverse possesses the same quality of smoothness. 

The obvious next step is to develop  \emph{quantitative} statements. What is
the ``size'' (norm) of the inverse element in the smooth subalgebra?
This is the problem of norm-controlled inversion. 
The quantification is of utmost importance in applied or numerical
constructions, where the control of the constants decides the success
or failure of a method. Naturally, these problems have been
investigated first in the context of the off-diagonal decay of
infinite matrices, see~\cite{Bas90}. It has been known for a long time that
an invertible matrix with exponential decay off the diagonal has an
inverse again with exponential off-diagonal decay~\cite{DMS84}, the
question of the precise constants involved turned out to be
fundamental for many applications. For polynomial off-diagonal decay
Baskakov has studied a subtle version of norm-control~\cite{Bas97}.

The term ``norm control'' is due to Nikolskii~\cite{Nik99}. He studied  ``the invisible
spectrum'' in the algebra of absolutely convergent Fourier
series and obtained estimates for the norms of inverses in this
algebra. Interestingly enough, the Fourier algebra is one of few examples so far, for 
which the lack of norm control  is known. As soon as one
passes to  absolutely  convergent \emph{weighted} Fourier series, one
can prove norm control~\cite{el-fallah02,Nik98}. 

In an  abstract context one may say that \emph{smoothness implies
inverse-closedness}. The topic of this paper is the much stronger
statement that \emph{smoothness implies norm control}. 

In this  paper and a subsequent paper  we will  study the concept of  norm control in smooth subalgebras
in a systematic manner. (a) What is smoothness in an abstract Banach
algebra? (b) Which forms of smoothness are preserved under inversion?
(c) Which forms of smoothness admit norm-control? 
The first two questions were studied and answered (at least
partially) in ~\cite{GK10}, the third question will be the central topic
of this  paper. 


In our main result we derive explicit   \nci\ for subalgebras defined
by   an extremely weak
condition of smoothness  that is expressed by  a differential (semi)
norm. This type of smoothness has been used heavily in operator
theory and non-commutative geometry~\cite{BC91,KS94,Rie10} and arises
naturally in approximation theory~\cite{DL93}.

\begin{thm} \label{main0}
  Assume that  \mB\ is a $C^*$-algebra and $\mA \subseteq \mB$ is a
  $*$-subalgebra with a common unit    and a 
  norm satisfying the inequality
\begin{equation}
\norm{ab}_\mA \leq C(\norm{a}_\mA \norm{b}_\mB+\norm{b}_\mA \norm{a}_\mB)\label{eq:diffsnint}
\end{equation}
for all $a, b \in \mA$. 
Then there exist constants $\gamma _j$ depending only the structure  constant
$C$, such that the following norm control estimates hold:

If $a\in \mA $ and $a$ is invertible in $\mB $ with condition number
$\kappa (a) = \norm{a}_\mB \norm{a\inv }_\mB \geq 5$, then 
  \begin{equation}
    \label{eq:17f}
    \norm{a \inv}_\mA \leq \max \Big\{ \gamma _1 \, \frac{\norm{a
      }_{\mA}}{\norm{a}^2_{\mB}} \,\, e^{\gamma _2\ln ^2 \kappa (a)} \,
      , \gamma _3 \, \frac{\norm{a
      }_{\mA}}{\norm{a}^2_{\mB}} \,\, e^{\gamma _4 \ln ^2 \Big( \gamma
      _5     \frac{\norm{a}_\mA ^2}{\norm{a}_\mB ^2}\Big)} \Big\} \, . 
  \end{equation}
\end{thm}

This result requires several commments. 

(i)  In Theorem~\ref{thm:asymptotics} we will derive explicit expressions
for the constants $\gamma _j$. 

(ii) Observe in particular that for $a\in \mA $ and $a$ invertible in
$\mB$, the inverse $a\inv $ is already in $\mA $. We say that $\mA $
is \IC\ in $\mB $.

(iii)  Though Theorem~\ref{main0} is a result
in operator theory, 
the proof requires methods from approximation theory and the
asymptotics of certain special functions. 

(iv) The case distinction in~\eqref{eq:17f} leads to an important
insight: the norm of $a\inv $ in $\mA $ does not only depend on the
condition number 
$\kappa (a)  $ in $\mB $ and $\norm{a}_\mA$, as one might have
expected from Nikolski's work~\cite{Nik99}, but also
on the embedding ratio $\norm{a } _\mA / \norm{a } _\mB $. This
phenomenon occurs already in the commutative case for $C^1(\bT )$. 

(v)   A norm on a  subalgebra  satisfying \eqref{eq:diffsnint} is called a
differential norm. This property can be considered a weak form of
smoothness in $\mB $. In ~\cite{Rie10}
Rieffel defined a strong Leibniz norm as a differential (semi)norm
satisfying the norm control $\norm{a\inv }_{\mA } \leq C \norm{a
}_{\mA }\norm{a\inv }_{\mB } ^2$ and asked whether every differential
seminorm is a strong Leibniz norm. The example of $n$-times
continuously differentiable functions on the torus $C^n(\bT )$ for
$n>1$ shows  that  this guess  cannot be true in general,
nevertheless Theorem~\ref{main0} states that a differential (semi)norm admits norm
control in a weaker sense. For applications in operator theory and
non-commutative analysis, e.g. for quantitative versions of the
holomorphic functional calculus, Theorem~\ref{main0} might be strong
enough.  

(vi) Theorem~\ref{main0} applies in particular to approximation algebras;
these are subalgebras of a given algebra that are defined by the 
approximation properties with respect to a so-called approximation
scheme or a filtration of subalgebras.

(vii) Note that \eqref{eq:17f} implies that $\norm{a\inv } _{\mA }
\leq \cO ( \kappa (a) ^{C \ln  \kappa (a)})$ for $\kappa \to \infty$. This means  that
$\norm{a\inv } _{\mA }$, as a function of the condition number,  grows
faster than every polynomial. 
 In the sequel~\cite{GK12b}  to this paper  we  will develop results for stronger
concepts of smoothness, namely for subalgebras of a given Banach
algebra with smoothness of Besov-H\"older-Lipschitz type of~\cite{GK10,Klo12}.  These
Banach algebras admit  a generalized  quotient rule, and  we will
prove much stronger statements about norm-control with polynomial
growth $\norm{a\inv } _{\mA }
\leq \cO ( \kappa (a) ^N) $  in the style of \eqref{eq:sont}. Since the
methods are completely different, we prefer to 
split our study of norm control into  two parts. 

The paper is organized as follows: In Section~2 we explain the
abstract  concept of norm control and give some equivalent
definitions. In Section~3 we prove that Banach algebras with a
differential norm admit \nci . Surprisingly, a standard construction
of  approximation theory always yields Banach algebras that admit \nci .  In Section~4 we   
 derive the   asymptotic estimate for the norm-controlling function
 of Theorem~\ref{main0}.  In the last section we will  briefly discuss
 other possible concepts of norm control and  the failure of \nci . At
this time the main example without norm control is the pair $\big(\mA (\bT ), C(\bT )\big)$
of absolutely convergent Fourier series and of continuous functions on
the torus~\cite{Nik99}.

\section{What is Norm-Controlled Inversion?}
\label{sec:norm-contr-invers-1}

In this section we will explore some facets of the abstract notion of
norm control.

\begin{defn} \label{dc:norm}
    Let $\mA \subseteq \mB$ be \BA s with common unit.

(i) We say that $\mA $ is inverse-closed in $\mB $, if every element
$a\in \mA $ that is invertible in $\mB $ is already invertible in $\mA
$, in short, $a\in \mA , a\inv \in \mB \, \Rightarrow \, a\inv \in \mA
$. 

(ii) We say that 
    $\mA$ admits \nci\ in \mB, if there is a function $h: (0,\infty
    )^2 \to (0,\infty )$ that
    satisfies 
\begin{equation*}
\norm{ a \inv}_\mA \leq h( \norm{a}_\mA , \norm{ a \inv}_\mB ) \, .
\end{equation*}
\end{defn}
Clearly, norm control is  a stronger property than
inverse-closedness.  If $\mA $ admits norm control in $\mB $, then, in particular, $\mA $
is \IC\ in $\mB $. 
 Inverse-closedness is a qualitative
property,  norm control is a quantitative property. The function $h$
provides an estimate for the norm of $a\inv $ in the smaller algebra
as a function of the norm of $a$ in the small algebra and the norm of
$a\inv $ in the large algebra.

 The control function $h$ in Definition~\ref{dc:norm} is not unique. To obtain a
well-defined control function, we use a slightly different definition
of norm control. 

\begin{lem}
  Let $\mA \subseteq \mB$ be \BA s with common unit element. Then
  \mA\ admits norm control, \fif\ 
there exists a function $\phi  : (0,1) \to (0,\infty )$, 
  such that,   for $a \in \mA$, 
  \begin{equation}
    \label{eq:3}
   \norm{a}_\mA \leq 1 \qquad \text{ and } \qquad  \norm{a \inv}_\mB \leq 1/\delta
  \end{equation}
implies that $a \inv \in \mA $  and 
$$\norm{a \inv}_\mA \leq \phi (\delta ) \, .
$$
 Explicitly, $\phi (\delta ) $ can be chosen to be 
\begin{equation} \label{eq:visconst}
  \phi(\delta)=\sup \set{\norm{a \inv}_\mA \colon \norm{a}_\mA \leq 1, \norm{a \inv}_\mB \leq 1/\delta} \,.
\end{equation}
The norm control function $h$ is then 
\begin{equation}
  \label{eq:c1}
  h(\norm{a}_\mA , \norm{a\inv} _\mB) =  \frac{1}{\norm{a} _\mA } \phi
  \Big( \frac{1}{\norm{a}_\mA  \norm{a\inv} _\mB}\Big)   \, .
\end{equation}
\end{lem}

\begin{proof}
  Given $a\in \mA $ with inverse $a\inv \in \mB $, we set $b=
  \tfrac{a}{\norm{a}_{\mA }}$. Then $\norm{b}_\mA = 1$ and
  $\norm{b\inv}_{\mB } = \norm{a}_\mA \norm{a\inv}_\mB = 1/\delta
  $. Then by definition of $\phi $ we have 
$$
\norm{a}_\mA \norm{a\inv}_\mA = \norm{b\inv}_{\mA } \leq \phi
(\frac{1}{\delta }) = \phi \big(\frac{1}{\norm{a}_\mA
  \norm{a\inv}_\mB}\big) \, ,
$$
and this is \eqref{eq:c1}. 
\end{proof}

The definition of norm control is adapted from
Nikolski~\cite{Nik99}. He studied inversion in  commutative Banach
algebras $\mA $ embedded in $C(X)$, the space of continuous functions  on
some compact set $X$. Then $\mA $ is called  $\delta $-visible for
fixed $\delta >0$, if there is a constant $ \phi (\delta) >0$, called the visibility
constant, such that 
 $\norm{f}_\mA \leq 1$ and $   \abs{f(x)} \geq \delta$ for all $x\in
 X$ implies that $\norm{ f\inv }_\mA \leq  \phi (\delta )$. 
Nikolskii determined the visibility constants explicitly for 
 the algebra of absolutely convergent Fourier series $\mA (\bT ) $.
 The visibility constants are  $\phi 
 (\delta ) = (2\delta ^2-1)\inv $ for $\delta > 1/\sqrt{2}$ and $\phi (\delta ) =
 \infty $ for $\delta \leq 1/2$, and seem to be unknown for $1/2\leq \delta \leq
 1/\sqrt{2}$. 
 Since  Definition~\ref{dc:norm}
 requires $\phi (\delta )$ to exist for every $\delta >0$, the Fourier
 algebra does not admit \nci .  
Related concepts were also studied in~\cite{Bjo72,Olof01}.


The following lemma contains several equivalent expressions for
Nikolski's visibility constant. 
\begin{lem} \label{condnec} 
  \begin{align*}
  \phi(\delta)&=\sup \set{\norm{a \inv}_\mA \colon \norm{a}_\mA \leq 1, \norm{a \inv}_\mB = 1/\delta} \\
              &=\sup \set{\norm{a \inv}_\mA \colon \norm{a}_\mA =1,  \norm{a \inv}_\mB  \leq 1/\delta} \\
              &=\sup \set{\norm{a}_\mA\norm{a \inv}_\mA \colon  \norm{a}_\mA \norm{a \inv}_\mB  \leq 1/\delta}
\end{align*}
\end{lem}
\begin{proof}
Set $A_\delta =  \set{a \in \mA \colon \norm{a}_\mA \leq 1, \norm{a
    \inv}_\mB = 1/\delta} $.  Obviously, 
\[ \sup \set{ \norm{a \inv}_\mA \colon a \in A_\delta} \leq
\phi(\delta) \, .\]
For the reverse inequality 
choose a sequence $a_n \in \mA$, such that 
$$\norm{a_n}_\mB \leq 1/\delta, \quad \norm{a_n}_\mA \leq 1 \qquad
\text{ and } 
\lim_n \norm{a \inv_n}_\mA = \phi(\delta) \,.
$$
Set  $\lambda_n = \delta \norm{a_n \inv }_\mB \inv  \geq 1$ and $
b_n = \lambda_n \inv  a_n$. Then $b_n \in A _\delta $,
because $\norm{b_n}_\mA \leq \norm{a_n}_\mA \leq 1$, $\norm{b_n\inv
}_\mB = \lambda _n \norm{a_n\inv }_\mB = 1/\delta $. Furthermore, 
\begin{align*}
 \sup \set{ \norm{a \inv}_\mA \colon a \in A_\delta} & \geq \liminf
 _{n\to \infty } \norm{b_n\inv}_{\mA }  \\
&= \liminf
 _{n\to \infty } \lambda _n \norm{a_n\inv}_{\mA } \geq \lim _{n\to
   \infty } \norm{a_n\inv}_{\mA } = \phi (\delta ) \, .
  \end{align*}

The second equality is proved with a similar argument. Set  $B_\delta
= \set{a \in \mA  \colon \norm{a}_\mA =1,  \norm{a \inv}_\mB  \leq
  1/\delta}$. Then  $\sup \set{ \norm{a \inv}_\mA \colon a \in B_\delta} \leq \phi(\delta) $. Using again the sequence $a_n$ defined above, we set
$\tilde a_n= \mu_n a_n$, where $\mu_n=\delta \norm{a \inv_n}_\mB \leq 1$. By construction, $\norm{ {\tilde a_n}\inv}_\mB=1/\delta$, $\norm{\tilde a_n}_\mA \leq 1$,
 and $\norm{ {\tilde a_n}\inv}_\mA\geq \norm{a \inv_n}_\mA$, which implies $\lim_{n \to \infty}\norm{ {\tilde a_n}\inv}_\mA = \phi(\delta)$.

The last equality is a just a reformulation of the second line.
\end{proof}

\begin{ex}
If $\mB $ is a $C^*$-algebra, then  the parameter $\delta $ can be interpreted as the smallest 
singular value of an element $a \in \mB $. For instance, a  \BA\ $\mA \subseteq
\bop$ admits \nci\ in \bop, if every operator  $A \in \mA $ that is
onto and satisfies $\norm{A}_\mA  \leq 1$ and $\|A c \|_2 \geq \delta \|c\|_2$
for all $c\in \ell ^2$,
satisfies also $\norm{A\inv
}_\mA \leq \phi (\delta )$. 
\end{ex}

\subsubsection*{Pseudospectrum}
\label{sec:pseudospectrum}

Recall that the $\delta$-pseudospectrum $\sigma_\delta^\mB(a)$ of an element $a$ in a unital \BA\ \mB\ is
\begin{equation}
  \label{eq:4}
  \sigma_\delta^\mB(a)= \set{\lambda \in \bc \colon
    \norm{(\lambda - a)\inv }_\mB > 1/\delta} \cup \sigma^\mB(a) \, .
\end{equation}


 Expressing the definition of $\phi $ from \eqref{eq:visconst} in
 terms of the pseudospectrum, we obtain the following reformulation of
 \nci .

\begin{lem}
A Banach algebra   $ \mA \subseteq \mB $  admits norm-controlled
inversion  if and only if there is a function $\phi $  such that for
all $\delta \in (0,1)$  and all   $a \in \mA$ with $\norm{a}_\mA =1$ 
\begin{equation}
  \label{eq:5}
   0 \notin \sigma_\delta^\mB(a) \quad \text{implies that }\quad  
 0  \notin \sigma_{1/\phi (\delta )}^{\mA}(a) \, .
\end{equation}
\end{lem}


\section{Norm-Controlled Inversion in Differential Subalgebras}
\label{sec:norm-contr-invers-4}

\subsection{ Differential Seminorms}
\label{sec:DiffSN}
A  general framework for describing smoothness in abstract  \BA s was
introduced in an  influential paper of  Blackadar and
Cuntz~\cite{BC91}.   Assume that  $\mA \subseteq
\mB$ are two \BA s with a common  unit element. A \emph{differential norm} (of order
1) is a norm on \mA\ that satisfies 
\begin{equation} 
\norm{ab}_\mA \leq C(\norm{a}_\mA \norm{b}_\mB +\norm{b}_\mA \norm{a}_\mB)\label{eq:diffsn}
\end{equation}
for all $a, b \in \mA$. In this case we call \mA\ a \emph{differential
  subalgebra} of \mB.   

Differential seminorms formalize a very general concept of smoothness.
Similar concepts   were introduced by Kissin
and Shulman\cite{KS94} and
Rieffel~\cite{Rie10}. A norm satisfying~\eqref{eq:diffsn} is also referred to
as a Leibniz norm on $\mA $. 

Let us recall a simple property of the constant occuring in differential
norms. We tacitly  assume that the  norms satisfy
$\norm{e}_\mA = \norm{e}_\mB   = 1$.

\begin{lem}
Assume that  $\mA \subseteq
\mB$ are two \BA s with a common  unit element and   let $\norm{\cdot
} _\mA $ a differential norm on $\mA $. Then either the norms
$\norm{\cdot } _\mA $ and $\norm{\cdot } _\mB $ are equivalent (and
$\mA $ is a closed subalgebra of $\mB $) or the structure constant $C$
in  \eqref{eq:diffsn} is  at least $1$. 
\end{lem}

\begin{proof}
  Equation~\eqref{eq:diffsn} with $b=e$ says that
$$
\norm{ae}_\mA \leq C(\norm{a}_\mA \norm{e}_\mB +\norm{e}_\mA
\norm{a}_\mB) = C(\norm{a}_\mA  + \norm{a}_\mB)   \, .
$$
Consequently, 
$$
\frac{\norm{a}_\mA}{\norm{a}_\mA}= 1 \leq C \Big( 1+
\frac{\norm{a}_\mB}{\norm{a}_\mA}\Big) \, .
$$
If $\norm{\cdot}_\mA $ is not equivalent to $\norm{\cdot }_\mB$, then
$\inf _{a\in \mA } \frac{\norm{a}_\mB}{\norm{a}_\mA} = 0$ and thus
\begin{equation*}
1\leq \inf C\Big( 1+   \frac{\norm{a}_\mB}{\norm{a}_\mA} \Big) = C  \, . 
\end{equation*}
\end{proof}
 In most applications the ambient algebra $\mB $ is a
 $C^*$-algebra. In this case the constant for the embedding  $\mA \subseteq \mB
 $ is also $1$, i.e., $\norm{a}_\mB \leq \norm{a}_{\mA }$ for all
 $a\in \mA $. 


As a warm-up we repeat Brandenburg's trick which shows  that a differential
subalgebra is \IC . For this  property we only need   
 that $\mB $ is a symmetric Banach algebra. This means that $\sigma
 _{\mB } (a^*a) \subseteq [0,\infty )$ for all $a\in \mB $.

\begin{lem}[\cite{BC91,bra75,KS94,GK10}] \label{brandlem}
  If $\mB $ is a symmetric unital  Banach algebra and  $\mA \subseteq \mB $ is
  a differential subalgebra with the same unit, then $\mA $ is \IC\ in $\mB $. 
\end{lem}

\begin{proof}
Applying (\ref{eq:diffsn}) with $a=b=c^n$ yields
\begin{equation} \label{eq:c3}
  \norm{c^{2n}}_\mA \leq 2C \norm{c^n}_\mA \norm{c^n}_\mB.
\end{equation}
Taking $n$-th roots and the limit $n \to \infty$ yields the inequality
of spectral radii $ \rho_\mA(c) \leq \rho_\mB(c)$ for all $c\in \mA $. Since the reverse
inequality is always true for $\mA \subseteq \mB$, we obtain the
equality of spectral radii. By a Lemma of Hulanicki~\cite{hulanicki} $\mA $ is inverse-closed in $\mB $.
  \end{proof}

The main step \eqref{eq:c3} above can be shaped into a quantitative
statement. We obtain the  result  that a  differential subalgebra
admits \nci . This is rather remarkable, because the condition of a
differential semi-norm is an extremely weak assumption.

\begin{thm}\label{thm-diffsubalg-nci}
  If  \mA\ is a differential $*$-subalgebra of a $C^*$-algebra
  \mB ,  then \mA\ admits \nci\ in \mB. For the norm of the inverse in
  $\mA $ we obtain  
\begin{equation}\label{eq:ncicstar}
  \norm{a \inv}_{\mA} \leq \frac{\norm{a}_\mA} {\norm{a}_{\mB}^2} 
\, \prod _{k=0}^\infty \Big( 1+  2 \frac{\norm{a}_\mA ^2}{\norm{a}_\mB
^2 } (2C)^k \,   \Big( 1- \frac{1}{\norm{a  }_{\mB }^2\norm{a \inv }_{\mB }^2}\Big) ^{2^k-1} \Big) 
\end{equation}
where  $C$ is the constant of (\ref{eq:diffsn}).
\end{thm}

\begin{proof} Equation (\ref{eq:diffsn}) implies that, for arbitrary $c\in \mA $ and
$k\in\bN$,  
we obtain
\begin{equation}
  \label{eq:itdsn}
  \norm{c^{2^k}}_\mA \leq 2C \norm{ c^{2^{k-1}}}_\mA \norm{
    c^{2^{k-1}}}_\mB \leq  2C \norm{c^{2^k-1} }_\mA \, \norm{c}_\mB
  ^{2^k-1} \,.
\end{equation}
  If we set 
\[
\beta_n =\beta_n(c)= \norm{c^n}_{\mA}\norm{c}_\mB^{-n} \,,
\]
then $\beta_n$ is \smult, and  (\ref{eq:itdsn}) says that $\beta
_{2^k} \leq 2C \beta _{2^{k-1}}$. Induction on $k$  implies that
\begin{equation}
  \label{eq:11}
  \beta_{2^k} \leq (2C)^k\beta_1 \,.
\end{equation}
Writing this inequality  as an inequality for the norms
of $c^{2^k}$, we obtain that
\begin{equation}
  \label{eq:cd1}
  \norm{ c^{2^k}}_\mA \leq \frac{\norm{c}_\mA}{\norm{c}_\mB}(2C)^k \,
  \norm{c}_\mB ^{2^k} \, . 
\end{equation}

Now let  $n=\sum_{k=0}^\infty \epsilon _k2^k$ be the
dyadic expansion of $n$   with digits   $\epsilon _k \in
\set{0,1}$. Let $\mF$ be the set of all sequences $\epsilon = (\epsilon _k) \in
\set{0,1}^\bN$ that contain only  finitely many $1$'s, then the dyadic
expansion is a bijection from $\bN $ onto
$\mF $. We now use the  submultiplicativity of $\beta _n$ and this
bijection to derive an estimate of $\norm{c^n}_\mA $ for arbitrary $n$, namely  
\begin{align}
    \norm{c^n}_{\mA } & = \norm{ \prod _{k=0}^\infty c^{2^k\epsilon
        _k}}_\mA \leq \prod _{k=0}^\infty \norm{c^{2^k}}_\mA
    ^{\epsilon _k}   \notag \\
&\leq  \prod _{k=0}^\infty \Big( \frac{\norm{c}_\mA}{\norm{c}_\mB} (2C)^k \,
  \norm{c}_{\mB } ^{2^k} \Big)^{\epsilon_k } \, . \label{eq:c4}
\end{align}
By summing up, we obtain  
\begin{align}
  \sum _{n=0}^\infty \norm{c^n}_\mA  &\leq \sum _{\epsilon \in \mF } \prod _{k=0}^\infty \Big( \frac{\norm{c}_\mA}{\norm{c}_\mB} (2C)^k \,
  \norm{c}_{\mB } ^{2^k} \Big)^{\epsilon_k } \notag \\
&= \prod _{k=0}^\infty \Big( 1+  \frac{\norm{c}_\mA}{\norm{c}_\mB} (2C)^k \,
  \norm{c}_{\mB } ^{2^k} \Big) \, .   \label{eq:cd2}
\end{align}
We note that this product converges, \fif\ 
$$
\sum _{k=0}^\infty (2C)^k \,
  \norm{c}_{\mB } ^{2^k} <\infty  \, ,
$$
which is the case, \fif\ $\norm{c}_\mB <1$. 
  
Now assume  that  $a \in \mA$ is invertible in \mB\  and set $b=
\frac{a^* a}{\norm{a^* a}_\mB}$. Then $b$ is hermitean,  invertible,
and $\norm{b}_\mB = 1$ and the spectrum $\sigma _\mB (b) $ is
contained in $
(0,1]$. Consequently, the spectrum of $c= e-b$ is
contained in an interval $\sigma _\mB (e-b) \subseteq [0,1-\epsilon ]
\subseteq [0,1)$, and in particular, $\norm{e-b}_{\mB } =
1-\epsilon < 1$. This implies that 
\[
b\inv = \sum_{k=0}^\infty (e-b)^k = \sum _{k=0}^\infty c^k 
\]
with convergence  in \mB . 

Consequently, the inverse of $a$ is given by 
\[
a \inv = \frac{ b \inv a^*}{\norm{a^* a}_\mB} \,.
\]
We now apply the estimate~\eqref{eq:cd2} to the element $c=e-b$ and obtain a
first estimate for the norm of $a\inv $ in \mA .
\begin{align*}\label{eq:ncisymmdiffalg}
 \norm{a \inv}_{\mA}  &\leq \frac{\norm{a^*}_\mA}{\norm{a^*a}_{\mB}}
 \Big(  \sum _{n=0}^\infty \norm{c^n}_{\mA } \Big) \\
& \leq \frac{\norm{a^*}_\mA}{\norm{a^*a}_{\mB}} 
\prod _{k=0}^\infty \Big( 1+  \frac{\norm{c}_\mA}{\norm{c}_\mB} (2C)^k \,
  \norm{c}_{\mB } ^{2^k} \Big)      \, .
\end{align*}
Finally we  estimate the norm of $c= e-\frac{a^* a}{\norm{a^* a}_\mB} $ directly by the
norms of $a$ and $a\inv $. 
Clearly
\begin{equation}
  \label{eq:norma}
  \norm{e-b}_\mA \leq \norm{e}_\mA + \norm{b}_\mA \leq 1 +
\frac{\norm{a^*a}_\mA}{\norm{a^*a}_\mB} \leq 1+ \frac{\norm{a}_\mA
  ^2}{\norm{a}_\mB^2} \leq 2  \frac{\norm{a}_\mA
  ^2}{\norm{a}_\mB^2} \, .
\end{equation}
On the other hand, since $\mB $ is a $C^*$-algebra  and $a^*a$ is
positive, we have $\norm{ (a^*a)\inv }_\mB ^{-1}  = \min \{ \lambda
:\lambda \in \sigma (a^*a)\}   = \lambda _{\min }$. Consequently, 
\begin{equation}
  \label{eq:normb}
  \begin{split}
  \norm{e-b}_\mB &= 1-\frac{\lambda _{\min} }{ \norm{a^*a}_\mB} \\
&=  1-\frac{1}{\norm{ (a^*a)\inv }_\mB \norm{a^* a}_\mB} =1-
\frac{1}{\norm{ a\inv }_\mB ^2 \norm{ a}_\mB^2} \, . 
\end{split}
\end{equation}
Inserting these estimates in~\eqref{eq:cd2}, we obtain the norm control
stated in~\eqref{eq:ncicstar}. 
\end{proof}
\begin{rems}
1. The proof is a modification of an approach of Sun~\cite{sun05}, the
important  estimate \eqref{eq:11} was   derived earlier  by
Blackadar and Cuntz~\cite[Lemma 3.8]{BC91} and Kissin and
Shulman~\cite{KS94}. Our main  contribution is the 
representation of the geometric series in~\eqref{eq:cd2}  as an infinite product, which
is amenable to a reasonable asymptotic analysis. 

2.   As $\norm{a^* a}_\mA \leq 2C \norm{a}_\mA \norm{ a}_\mB $, we can
replace the inequality (\ref{eq:norma}) by 
\[
\norm{e-b}_\mA \leq (1+2C)\frac{\norm{a}_\mA}{\norm{a}_\mB}
\]
If $\norm{a}_\mA / \norm{a}_\mB \geq 1/2 + C$, one obtains a slightly
better estimate in \eqref{eq:ncicstar}. 
\end{rems}
 
\subsection{Application: General Approximation Spaces}

In approximation theory differential (semi-) norms appear naturally in the
definition of approximation spaces~\cite{DL93}. 

 An \emph{approximation scheme} on the \BA\ \mA\
is a family $(X_n)_{n \in \bn_0}$ of closed subspaces of \mA\ that satisfy
$ X_0=\set{0}$, $ X_n \subseteq X_m$  for $n \leq m$,  and
$  X_n \cdot X_m\subseteq X_{n+m}$, $n,m\in \bn_0$.
If \mA\ is a $*$-algebra, we  assume that
$ e \in X_1$  and $ X_n=X^*_n$ for all $ n \in \bn_0$.
The \emph{$n$-th approximation error} of  $a \in \mA$ by $X_n$ is
$  E_n(a)=\inf_{x \in X_n} \norm{a-x}_\mA$.
For  $1\leq p < \infty $ and $w$ a weight on $\bno$ the approximation space $\app p w  \mA$  consists of all $a \in \mA$ for which the norm 
\begin{equation}
  \label{eq:appspace}
  \norm{a}_{\mE_w^p}= \bigl(\sum_{k=0}^\infty {E_k(a)^p}w(k)^{p} \bigr)^{1/p}
 \end{equation}
is finite (with the standard modification for $p=\infty$).

Algebra properties of \as s are discussed
in~\cite{AL06,GK10}. The invertibility of elements in an
approximation space was investigated in~\cite{GK10}. In particular,
we proved  the
following result.

\begin{prop} \label{appspIC} If \mA\ is a symmetric \BA\ with
  approximation scheme $(X_n)_{ n \in \bn_0}$ and $w$ is a subadditive
  weight function on $\bN_0 $, then $\mE_w^p(\mA)$ is \IC\ in $\mA$.
\end{prop}

The decisive inequality in the proof of this proposition was 
\begin{equation}
  \label{eq:10}
  \norm{a b}_{\app p w \mA} \leq C_w (\norm{a}_\mA \norm{b}_{\app p r
    \mA}+\norm{b}_\mA \norm{a}_{\app p r \mA}) \, . 
\end{equation}
In other words, the norm of   $\app p w \mA$ is
a differential norm in  \mA. (In~\cite{GK10} Proposition~\ref{appspIC}
was stated only for polynomial weights $v(k) = k^r$, the statement and
proof is identical for general subadditive weights.) 

As a consequence of Theorem~\ref{thm-diffsubalg-nci} we obtain the following statement, which came as a
surprise to us.

\begin{cor}
  Assume that $\mA $ is a $C^*$-algebra with an approximation scheme
  $(X_n)_{n\in \bN _0}$ and that $w$ is a subadditive weight on $\bN_0
  $. Then the approximation algebra $\app p w \mA$ admits \nci . 
\end{cor}
\begin{rems}
(1)  In approximation theory one uses  the  polynomial weights
 $w(k)=v_r(k)= k^r$ for  $r>0$.  In the literature, the 
resulting approximation space  $\app p {v_{r-1/p}} \mA$ is usually
denoted as $\app p r \mA $, in contrast to our notation. 

(2) In many  situations more structure is available, e.g., a derivation
  or an  automorphism group on $\mB $,  and   the
  approximation space $\app p r \mA$ can be identified with a  Besov
  space or some  other space, see \cite{DL93,Klo12, Klo12a}.

In this case the estimates for
  norm control can be  improved  significantly, even without the
  assumption that the ambient algebra is a $C^*$-algebra. This  will be the
  topic of Part II~\cite{GK12b}. 
\end{rems}


\section{Asymptotics of norm-controlled inversion in differential subalgebras}
\label{sec:asympt}

In principle, the estimate in \eqref{eq:ncicstar}  yields a norm-controlling
funcion $h$, but the expression is too cumbersome to decode any useful
information. Our main effort is to  derive  a clean asymptotic expression
for the norm controlling function. The following theorem offers a  norm controlling function   
in which all constants are explicit and depend only on the structure
parameter of the differential norm.

\begin{thm} \label{thm:asymptotics}
  Assume that  \mB\ is a $C^*$-algebra and $\mA \subseteq \mB$ is a
  differential $*$-subalgebra with a common unit and a differential
  norm satisfying \eqref{eq:diffsn} with the structure constant $u=2C
  $. 

Set  
\begin{align*}
  K & = (\ln 2 - 1/2)\inv \approx 5.1774 \, , \\
\gamma _1   &= e \, \exp \Big( \frac{8 \ln u \ln ^2 (\ln u )}{\ln ^2
    2} \Big) \qquad \text{ and } \qquad 
\gamma _2 = \frac{16 \ln u}{\ln ^2 2}  \,,\\
\gamma _3 &= e \, \exp \Big( \frac{8\ln ^2 K}{\ln u} \Big) \qquad
\text{ and } \qquad 
\gamma _4 = \frac{4}{\ln u} \, . 
\end{align*}

If $\kappa  (a) \geq 5$, then 
  \begin{equation}
    \label{eq:16}
    \norm{a \inv}_\mA \leq \max \Big\{ \gamma _1 \, \frac{\norm{a
      }_{\mA}}{\norm{a}^2_{\mB}} \,\, e^{\gamma _2\ln ^2 \kappa
      (a)^2} ,  \,  \gamma _3 \, \frac{\norm{a
      }_{\mA}}{\norm{a}^2_{\mB}} \,\, e^{\gamma _4 \ln ^2 \Big( 2 u^{16}
    \frac{\norm{a}_\mA ^2}{\norm{a}_\mB ^2}\Big) } \Big\} \, . 
  \end{equation}
If $\kappa (a)^2 \geq (\ln u)\inv \big( 10 \norm{a}_\mA ^2
/\norm{a}_\mB ^2 \big)^{\ln 2/\ln u}$, then always
$$    \norm{a \inv}_\mA \leq  \gamma _1 \, \frac{\norm{a
      }_{\mA}}{\norm{a}^2_{\mB}} \,\, e^{\gamma _2\ln ^2 \kappa
      (a)^2} \, .
$$
\end{thm}

The following corollary provides a simpler description and an explicit
norm controlling function $h$ in the sense of Definition~\ref{dc:norm}.

\begin{cor}
If  $\mA $ is a
  differential $*$-subalgebra of a unital  $C^*$-algebra $\mB $ with a
  common unit, then there exist constants $C_1, C_2 > 0$   such that 
  $$
\norm{a\inv } _{\mA } \leq C_1 \norm{a}_{\mA } \norm{a\inv }_{\mB } ^2
\, e^{C_2 \ln ^2 ( \norm{a}_{\mA } \norm{a\inv }_{\mB } )} \, .
$$
Thus the controlling function may be taken to be 
$$
h(x,y) = C_1xy^2 e^{C_2 \ln^2 (xy)} \, .
$$
\end{cor}

\begin{proof}
The norm control estimates in Theorem~\ref{thm:asymptotics} depends
on the three parameters $\norm{a}_{\mA }, \norm{a}_{\mB } $ and $ \norm{a\inv
}_{\mB }  $. We eliminate $\norm{a}_{\mB }$  by using the embedding
inequality $ \norm{a }_{\mB } \leq  \norm{a }_{\mA }$ and the Banach
algebra inequality $\kappa (a) =   \norm{a }_{\mB }   \norm{a\inv
}_{\mB } \geq 1$, whence $ \norm{a }_{\mB }\inv \leq  \norm{a\inv
}_{\mB }  $. Then both $\kappa (a) \leq   \norm{a }_{\mA }   \norm{a\inv
}_{\mB }$ and $ \frac{\norm{a}_\mA }{\norm{a}_\mB } \leq  \norm{a }_{\mA }   \norm{a\inv
}_{\mB }$. After adapting the constants in \eqref{eq:16}, 
we obtain the stated inequality.  
\end{proof}

\begin{rem}
  The case distinction in  \eqref{eq:16} 
  reveals an important and completely 
  new aspect of \nci . The norm of $a\inv $ in the smaller algebra
  $\mA $ may depend on both  the condition number $\kappa 
  (a) $ in the larger algebra $\mB $ and on  the embedding
parameter  $\tfrac{\norm{a}_{\mA}}{\norm{a}_{\mB}}$, whichever number
  dominates. 

In the extreme case of a unitary element in $\mB $ we have $a\inv =
a^*$, so that $\kappa (a)  = 1$ and $\norm{a\inv }_{\mA} =
\norm{a}_\mA $ depends only on the embedding ratio
$\tfrac{\norm{a}_{\mA}}{\norm{a}_{\mB}}$. 

As a second example we look at the algebra $\mA =C^1(\bT ) $ of continuously
differentiable functions on the torus embedded in $\mB = C(\bT )$. Consider
the sequence 
$$
a_n(t) = 1 + \tfrac{1}{2} \cos 2\pi n t \, .
$$
Then $1/2 \leq a_n(t) \leq 3/2$ and $\kappa (a)  = 3$, whereas 
$$
\frac{\norm{a_n}_{C^1}}{\norm{a_n}_{C}} = \frac{3+2 \pi n}{3} \, .
$$
In this case the norm of $a_n\inv $ in $C^1$ is of order
$$
\norm{a_n\inv } _{C^1} \geq 2\pi n \, ,
$$
and is completely determined by the ratio
$\frac{\norm{a_n}_{C^1}}{\norm{a_n}_{C}}$ and not by the condition
number of $a_n$ in $C(\bT )$. 
\end{rem}

The remainder of this section is devoted to proving the asymptotics. 
In view of \eqref{eq:ncicstar} we introduce the function
\begin{equation}
  \label{eq:cd3}
  f(u,v,c) = \prod _{k=0}^\infty (1+cu^k v^{2^k}) \, .
\end{equation}
We are interested in the asymptotic behavior of $f$ as
$v$ tends to $1$ for fixed $u$ and  want to express the constants  in
terms of $u$ as explicitly as possible. Setting
$$
u= 2C, \quad v =  1- \frac{1}{\kappa (a)^2}, \quad \text{ and } c =    2 \frac{\norm{a}_\mA ^2}{\norm{a}_\mB
^2 }   \big( 1- \frac{1}{\norm{a  }_{\mB }^2 \norm{a \inv }_{\mB
}^2}\big) ^{-1}
$$
Theorem~\ref{thm-diffsubalg-nci} says that 
$$
  \norm{a \inv}_{\mA} \leq \frac{\norm{a}_\mA} {\norm{a}_{\mB}^2} 
\, f\Big(u,v,c  \Big)  \, .
$$

We will prove   the following  estimate for $f(u,v,c)$.

\begin{prop} \label{asf}
  Assume that  $\ln u /\ln v\inv \geq 16$ and set $K= (\ln 2 - 1/2)\inv
  $. Then 
  \begin{equation}
    \label{eq:cd18}
    f(u,v,c) \leq 
    \begin{cases}
      e \,   \exp \big( \frac{8 \ln u \ln ^2 (\ln u)}{\ln
      ^2 2}\big) \, \exp \big( \frac{8 \ln u}{\ln ^2 2} \, \ln ^2
    \big(\frac{1}{\ln v\inv} \big)\big)\,  \quad & \text{ if } \ln u/\ln
    v\inv \geq (Kc)^{\ln 2/\ln u} \, \vspace{3 mm} \\
     e \exp \Big(  \frac{8\ln ^2 K}{\ln u} \Big) \, \exp \big( \frac{4}{\ln u} \, \ln ^2
   c\big) \, , \quad & \text{ if } \ln u/\ln v\inv < (Kc)^{\ln 2/\ln
     u} \, .
    \end{cases}
  \end{equation}
\end{prop}

\vs

\begin{proof}[Proof of Theorem~\ref{thm:asymptotics}] Our main theorem
    follows from Proposition~\ref{asf}
by substitution. We keep in mind that 
 $u = 2C \geq 2$ is the structure
constant of the differential algebra and $v= 1- \kappa (a)^{-2}< 1$.
In particular  the  principal parameter $(\ln v\inv )\inv $
satisfies 
\begin{equation}
  \label{eq:cd13}
  \frac{1}{\ln v\inv} =  -  \frac{1}{\ln (1- \kappa (a) ^{-2})} \leq
  \kappa (a) ^2 \, , \text{ for all } \kappa (a) \geq 1 \, ,  
\end{equation}
and 
$
\lim_{ \kappa (a) \to \infty} -{\ln (1- \kappa (a) ^{-2})} \kappa (a) ^2 =1
$. 
The restriction $\kappa  (a) \geq 5$ comes from the assumption $\ln
u/\ln v\inv \geq 16$ in Proposition~\ref{asf}. Solving for $\kappa (a)$ yields
$$
\kappa  (a)  \geq \big( 1- u^{-1/16}\big)^{-1/2} \approx 4.857 \, ,
$$
and since $u = 2C \geq 2$, we may take $\kappa (a) \geq 5  >
(1-2^{-1/16})^{-1/2}$. 

Note that 
$$c =  2\frac{\norm{a}_\mA ^2}{\norm{a}_{\mB }^2}
\frac{1}{1- \kappa (a)^{-2}}  \geq 1 \, ,
$$
 depends on $a$, but not on the index $k$ in  the infinite
product~\eqref{eq:cd3}.  

The restriction for the special case follows from  making the condition
$\ln u/\ln
    v\inv \geq (Kc)^{\ln 2/\ln u}$ in \eqref{eq:cd18} explicit. We
    have 
$$
\kappa (a) ^2 \ln u \geq \frac{\ln u}{\ln v\inv } \geq \Big(2 K \frac{\norm{a}_\mA ^2}{\norm{a}_{\mB }^2}
\frac{1}{1- \kappa (a)^{-2}}\Big)^{\ln 2/\ln u} \geq \Big(10
\frac{\norm{a}_\mA ^2}{\norm{a}_{\mB }^2} \Big)^{\ln 2/\ln u} \, .
$$
\end{proof}

We carry out the proof of Proposition~\ref{asf} in several steps with intermediate lemmas. 
\vs

\textbf{Step~1.} \emph{Determine the maximum of $1+cu^k v^{2^k}$.}

 The asymptotics of $f(u,v,c)$ for $v\to 1$ will depend on the value of
 $u$. In the following we will often need the logarithmic ratio of $u$
 and $v$, and so we define the number $\xi $ by 
 \begin{equation}
   \label{eq:cd4}
   2^\xi = \frac{\ln u}{\ln v\inv}\quad \text{ and } \xi = \log _2
   \Big( \frac{\ln u}{\ln v\inv} \Big) \, ,
 \end{equation}
where $\log _2 z = \ln z/\ln 2 $ is the logarithm with basis $2$. We
will need $\xi \geq 4$ and thus $2^\xi \geq 16$ in the proof. 
Note that since $u= 2C \geq 2$ and $v<1$, $\xi \geq 4 $ is
well-defined and that 
$\lim _{v\to 1-} \xi = \infty $. 
Furthermore
\begin{equation}
  \label{eq:cd6}
v^{2^\xi } =  \exp \Big( \ln v \frac{\ln
  u}{\ln v\inv} \Big) = u\inv \, .  
\end{equation}

\begin{lem} \label{maxlem}
  If $u\geq 1$ and $v<1$, then 
  \begin{equation}
    \label{eq:cd5}
    1+cu^k v^{2^k} \leq 1+cu^\xi  \qquad \forall k\in \bN \, .
  \end{equation}
\end{lem}

\begin{proof}
  Let $c_k = c u^k v^{2^k}$. Then 
$$
\frac{c_{k+1}}{c_k} = \frac{cu^{k+1} v^{2^{k+1}}}{cu^k v^{\dyad }} = u
v^{\dyad } \, ,
$$
and $1+ c_k$ is decreasing, \fif\ $uv^{\dyad }\leq 1$. Thus $1+c_k$ is
decreasing, \fif\ $\dyad
\geq \tfrac{\ln u}{\ln v\inv} = 2^\xi $.  Set $K=\lceil \xi \rceil $. Then
$1+c_k$ is decreasing for  $k\geq K=\lceil \xi \rceil
$, increasing for $k< K$, and  the maximum of $c_k$ is taken at
$K$. Writing $K= \lceil \xi \rceil  = \xi + \eta $ for
some $\eta \in [0,1)$ and using~\eqref{eq:cd6}, we find that 
\begin{align*}
u^K v^{2^K} &= u^{\xi + \eta } \big( v^{2^\xi  }\big) ^{2^\eta } \\
&= u^{\xi +\eta -2^\eta } \leq u^\xi \, ,    
\end{align*}
as claimed. (In fact $u^{\xi +\eta -2^\eta } \leq u^{\xi - 0.91}$.)
\end{proof}

\vs

\textbf{Step~2.} \emph{ Split the infinite product \eqref{eq:cd3} into a finite product $\prod
_{k=0}^M$ and a remainder  $\prod _{k=M+1}^\infty $ and estimate the
remainder first.}

We first 
determine a  cut-off index $M$,  so
that $\prod _{k=M+1}^\infty (1+cu^k v^{\dyad }) \leq e$. 

\begin{lem}
  Let $K=\frac{1}{\ln 2 - 1/2}$ and choose $M$, such that  
  \begin{equation}
    \label{eq:cd14}
    M+1  \leq \xi + 2 \ln _2 \Big( \max (\xi, \frac{\ln (Kc)}{\ln
    u})\Big) \leq M+2 \, . 
  \end{equation}
Then 
$$
 \prod _{k=M+1}^\infty (1+cu^k v^{2^k}) \leq e \, .
$$ 
\end{lem}

\begin{proof}
  Using the estimate
$$
1+cu^kv^{\dyad } \leq \exp \Big(cu^kv^{\dyad }\Big)\, ,
$$ 
the remainder term is bounded by 
$$ 
    \prod _{k=M+1}^\infty (1+cu^k v^{2^k}) \leq \exp \Big(
    \sum _{k=M+1}^\infty cu^k v^{2^k} \Big) \, .
$$
By Step~1  the sequence $k\to u^k v^{\dyad}$ is
  decreasing for $k\geq M \geq \xi$, and we 
estimate the occurring sum by the integral
$$
\sum  _{k=M+1}^\infty u^k v^{2^k} \leq \int _M ^\infty u^x v^{2^x} \, dx \, .
$$ 
After the substitution $y= 2^x \ln v\inv $ ( precise details: $ x= \log _2 (y/\ln
v\inv )$ and $dx= \tfrac{1}{\ln 2 } \, \tfrac{dy}{y}$, $u^x =
\Big(\tfrac{y}{\ln v\inv }\Big)^{\log
  _2 u} $) the integral turns into 
\begin{equation}
  \label{eq:cd12}
  \int _M ^\infty u^x v^{2^x} \, dx = \frac{1}{\ln 2}\, \frac{1}{(\ln v\inv)^{\log _2
      u}} \int _{2^M \ln v\inv } ^\infty y ^{\log _2 u  } e^{-y} \,
  dy \, .
\end{equation}
The new integral is an  incomplete  Gamma-function defined as 
$$
\Gamma (a,x) = \int _x ^\infty t^{a-1} e^{-t} \, dt  = x^{a-1} e^{-x} \Big(
1+\sum _{j=1}^{n-1} \frac{(a-1)(a-2) \dots (a-j)}{x^j} + R_n(x,a) \Big) \, .
$$
For $n\geq a-1 $ the remainder $R_n$  satisfies the 
estimate
 $R_n(x,a) \leq \frac{|(a-1) \dots
   (a-n)|}{x^n}$. See~\cite[8.11(i)]{AS64} or 
 verify directly with integration by parts. 

In our case  $a= \log _2 u +1$ and $x=2^M \ln v\inv   \geq 2 \cdot 2^\xi \ln
v\inv   = 2\ln u$. If  we choose  $M\geq \xi +1$ and $n$ such that d $a-1 \leq n < 2a
+1$, then 
$$
0\leq \frac{|a-j|}{x}\leq \frac{a-1}{x} \leq \frac{\log _2 u}{2^M \ln v\inv
} \leq \frac{1}{2\ln 2} = \frac{1}{\ln 4} <1 \, ,
$$
and consequently
$$
1+ \sum _{j=1}^{n-1} \frac{(a-1)(a-2) \dots (a-j)}{x^j} + R_n(x,a)
\leq 1+ \sum
_{j=1}^n \frac{1}{(\ln 4)^j} \leq \frac{1}{1- (\ln 4)\inv} \, ,
$$
which is a bound independent of $x$ and $a$ and $n$. 

Combining these estimates, we find that 
\begin{align*}
   \int _M^\infty u^x v^{2^x}\, dx &= \frac{1}{\ln 2 (\ln v\inv
    )^{\log _2 u}} \Gamma (\log _2 u +1, 2^M \ln v\inv ) \\
& \leq \frac{1}{\ln 2 (1 - (\ln 4)\inv)} \frac{(2^M \ln v\inv ) ^{\log
  _2 u}}{(\ln v\inv )^{\log _2 u}} \, \exp \Big(-2^M \ln v\inv \Big)
\\
&= K  u^M \exp \Big(-2^M \ln v\inv \Big) \, .
\end{align*}
Here we set $K= \Big(\ln 2 (1 - (\ln 4)\inv)\Big)\inv = \Big(\ln 2 -
1/2\Big)\inv = 5.177 >1$. 
Now we choose $M\geq \xi +1$ so that 
\begin{equation}
  \label{eq:cd16}
   K c u^M \exp \Big(-2^M \ln v\inv \Big) \leq 1 \, .
\end{equation}
Then
$$
\prod _{k=M+1}^\infty (1+cu^k v^{2^k}) \leq  \exp \Big(
    \sum _{k=M+1}^\infty cu^k v^{2^k} \Big) \leq \exp \Big( Kc u^M
    \exp \Big(-2^M \ln v\inv \Big) \Big) \leq e \, .
$$
For an explicit estimate of $M$ we note that \eqref{eq:cd16} is
equivalent to the estimate
\begin{equation}
  \label{eq:cd17}
  2^M  \ln v \inv - M \ln u  \geq \ln (Kc) \, .
\end{equation}
By  writing $M = \xi +\lambda $ and  using $2^\xi = \ln u /\ln
v\inv $, we obtain that
$$
2^{\lambda } - \lambda \geq \xi    + \frac{\ln (Kc)}{\ln u} \, .
$$
Again since $\xi \geq 4$,  we may use
$$
2^\lambda - \lambda \geq 2^{\frac{\lambda}{2}+1} \geq 2 \max \Big( \xi
, \frac{\ln (Kc)}{\ln u} \Big)
$$
and we obtain that $\lambda \geq 2 \log _2 \max \Big( \xi
, \frac{\ln (Kc)}{\ln u} \Big)$. 
We find that 
 $$M+1  \leq \xi + 2 \log _2 \Big( \max (\xi, \frac{\ln (Kc)}{\ln
    u})\Big) \leq M+2 \,  
 $$ 
is a  workable choice for $M$. 
\end{proof}

\textbf{Step~3.} \emph{Estimate the finite part of $\prod _{k=0}^M  $.}
We use Lemma~\ref{maxlem} and the obvious estimate
$$
\prod _{k=0}^M (1+cu^k v^{2^k}) \leq \max _{k=0, \dots , M} (1+cu^k
v^{2^k})^{M+1} \leq (1+cu^\xi )^{M+1} \, .
$$
Since the definition of $M$ contains a $\max $, we distinguish two
cases. 

\textbf{Case~1.} If $\xi \geq \frac{\ln (Kc)}{\ln u}$ or,
equivalently,  $Kc  \leq u^\xi
$, then 
$$
M+1 \leq \xi + 2 \log _2 \xi \leq 2\xi 
\, 
$$
and, since $\xi \geq 4$,  $c\geq 1$, and $u\geq 2$,  
$$
\frac{M+1}{cu^\xi } \leq \frac{2\xi}{cu^\xi } \leq 1 \, .
 $$

Then 
$$
\Big( 1+ \frac{1}{cu^\xi }\Big) ^{M+1 } \leq  \exp \Big( \frac{M+1 }{cu^\xi } \Big) \leq \exp (1) \, ,
$$
and, since $K\geq 1$ and $\ln ^2 (a/b) \leq 2(\ln ^2 a + \ln ^2 b\inv )$,  
\begin{align}
  (cu^\xi )^{M+1}  &\leq \Big(\frac{u^{2\xi }}{K}\Big)^{M+1} \leq u^{4
    \xi ^2}  \label{finalshit} \\
&\leq \exp \Big( 4  \frac{\ln u}{\ln ^2 2} \, \ln ^2 \big( \frac{\ln
  u}{\ln v\inv} \big) \Big) \notag  \\
&\leq \exp \Big(  \frac{8 \ln u \ln ^2 (\ln u )}{\ln ^2 2} \Big) \, \exp \Big(
\frac{8 \ln u}{\ln ^2 2} \, \ln ^2 \big( \frac{1}{\ln v\inv} \big) \Big)
\, .  \notag 
\end{align}
Thus the final estimate is
$$
\Big( 1+ cu^\xi \Big)^{M+1} = \Big( 1+ \frac{1}{cu^\xi }\Big) ^{M+1 }
(cu^\xi )^{M+1} \leq C(u) \exp  \Big(
\frac{8 \ln u}{\ln ^2 2} \, \ln ^2 \big( \frac{1}{\ln v\inv} \big) \Big)
$$
with the constant $C(u) = e \exp \Big( \frac{8 \ln u \ln ^2 (\ln
u  )}{\ln ^2 2} \Big)$ that depends only on $u$. 

\textbf{Case~2.} If $\xi \leq \frac{\ln (Kc)}{\ln u}$ or
equivalently, $Kc  \geq u^\xi
$, then 
$$
M+1 \leq \xi +2\log _2 \frac{\ln (Kc)}{\ln u} \leq 2 \frac{\ln
  (Kc)}{\ln u} \, .
$$
In this case we obtain similarly
\begin{align*}
  (1+cu^\xi ) ^{M+1} &\leq (1+Kc^2)^{M+1} \leq (Kc^2)^{2  \frac{\ln
  (Kc)}{\ln u}} (1+ \frac{1}{Kc^2})^{2  \frac{\ln
  (Kc)}{\ln u}} \\
&\leq e \, \exp \Big( \ln (Kc^2) \, 2 \, \frac{\ln
  (Kc)}{\ln u} \Big) \\
&\leq e \exp   \Big(   \, \frac{2}{\ln u}  \ln ^2 (Kc^2) \Big) \\
&\leq C'(u) \exp \Big(   \, \frac{8}{\ln u}  \ln ^2 c \Big)
\end{align*}
with a constant $C'(u) = e \exp \Big(  \frac{4\ln ^2 K}{\ln u} \Big)$.  

The proof of Proposition ~\ref{asf} is now finished.  \hfill $\Box$

\begin{rem}
  The proof shows that the asymptotic estimates in Theorem~\ref{thm:asymptotics} are in fact
  \begin{equation}
    \label{eq:hm25}
\norm{a \inv}_\mA = o\Big( \frac{\norm{a
      }_{\mA}}{\norm{a}^2_{\mB}} \,\, e^{\gamma _2\ln ^2 \kappa
      (a)^2} \Big) \, \qquad \text{ for } \kappa (a) \to \infty \, .
      \end{equation}
In  fact, in \eqref{finalshit} we have given away a factor 
$$
K^{-M-1}\leq K^{-\xi } \leq K^{-2 \ln \kappa (a) \ln \ln u } \, .
$$   
Taking this factor into account, we obtain the slightly stronger
statement of \eqref{eq:hm25}.
\end{rem}

\section{Related  Concepts of Norm Control and Lack of Norm Control} 
\label{sec:rsc}

In our definition of \nci\ we have used only the two quantities
$\norm{a}_\mA $ and $\norm{a\inv }_\mB $ to control
$\norm{a\inv}_{\mA} $. By adding more parameters, one is lead to
alternative concepts of \nci . 

(i) Sun~\cite{sun05,Sun07a} considers a nested pair of Banach algebras $\mA
\subseteq \mB $ with a modified differential norm satisfying 
\begin{equation}
  \label{eq:sun}
   \norm{a^2}_\mA \leq 2C \norm{a}_\mA^{1+\theta} \norm{a}_{\mB}^{1-\theta} \,
\end{equation}
for a fixed exponent $\theta , 0< \theta <1$. Versions of Lemma~\ref{brandlem} and
Theorem~\ref{thm-diffsubalg-nci} hold for such norms. In fact, our
proof is a modification of~\cite{sun05}. Probably the asymptotic
estimates of Theorem~\ref{main0} could also be  adapted to this case. Norms
satisfying~\eqref{eq:sun} are important in the study of algebras of
infinite matrices with off-diagonal decay.


(ii) In algebras with an approximation scheme 
one may also  include an approximation
parameter to control the  norm $\norm{a\inv }_\mA $. This fundamental
idea is used first in  the deep work of Baskakov on the inversion of matrices with off-diagonal
decay~\cite{Bas97}.  To give an explicit example we formulate Theorem~4
of~\cite{Bas97} for the Wiener algebra $\mA (\bT )$. 
Let $a(t) = \sum _{k\in \bZ } a_k e^{2\pi i k t}$ with norm
$\norm{a}_{\mA } = \sum _k |a_k|$ and define the tail function of $a$
by 
$$
\psi _a (k) = \sum _{|j| \geq k } |a_j| \, .
$$
If $a\in \mA $ and $a(t) \neq 0$ everywhere, then 
$$
\norm{a\inv }_{\mA } \leq 64 \norm{a}_\infty \norm{a\inv }_\infty ^2
\psi _a \Big(\frac{1}{4+32 \norm{a}_\infty \norm{a\inv }_\infty ^2} \Big) \,
.
$$
In contrast to Definition~\ref{dc:norm} this estimate depends not only on the
condition number $\norm{a}_\infty \norm{a\inv}_\infty$, but also on the
rate of approximation of $a$ by trigonometric polynomials, as
expressed by the tail function $\psi $.  

Similar estimates  can be found in  Tao's   quantitative version of Wiener's
Lemma~\cite{Taowiener}.

\vs

\textbf{Lack of Norm Control.} 
Currently there are extremely few examples of Banach subalgebras
without  norm control. The fundamental  example is the algebra $\mA = \mA (\bT )
$ of absolutely convergent Fourier series as a subalgebra of the
algebra $C(\bT )$ of continuous functions on the
torus~\cite{Nik99}. In his investigation of the corona problem  Nikolski
constructed further and rather sophisticated examples of multiplier
algebras that lack norm control~\cite{Nik09}.

The example 
$\mA (\bT ) \subseteq C(\bT )$ can be generalized to group
algebras~\cite[1.2.3]{Nik99}. Precisely, let  $G$  be a  discrete group that contains at
least one  torsion free element. We choose  the  $\ell ^1$-algebra
$\mA = \ell ^1(G)$ with respect to convolution and  $\mB = 
C^*(G)$ the enveloping $C^*$-algebra. Then $\ell ^1 (G)$ always lacks
norm control in $C^*(G)$. Since $\bZ $ is a subgroup of $G $, 
$\ell ^1(\bZ ) \simeq \mA (\bT )$ is a closed subalgebra of $\ell
^1(G)$. Likewise, $C(\bT )$ can be identified with  a closed
subalgebra of $C^*(G)$. Consequently the lack of norm control follows
immediately from the example $(\mA (\bT ), C(\bT)$.


\def\cprime{$'$} \def\cprime{$'$} \def\cprime{$'$} \def\cprime{$'$}
  \def\cprime{$'$}


\end{document}